\documentclass[11pt]{article}
\usepackage[a4paper, margin=1.2in]{geometry}
\usepackage{setspace}
\usepackage{amsmath,amssymb,amsthm,mathtools}
\usepackage{enumitem}
\usepackage{xcolor}
\usepackage{microtype}
\usepackage[colorlinks=true,linkcolor=blue,citecolor=blue,urlcolor=blue]{hyperref}
\usepackage{comment}
\usepackage{tikz}
\usetikzlibrary{calc, positioning, decorations.pathmorphing, arrows.meta, shapes.geometric, shadows, backgrounds}
\pgfdeclarelayer{background}
\pgfsetlayers{background,main}
\allowdisplaybreaks
\definecolor{jctbBlue}{RGB}{70, 130, 180}     
\definecolor{jctbOrange}{RGB}{255, 127, 80}   
\definecolor{pathBlack}{RGB}{40, 40, 40}      
\definecolor{logicRed}{RGB}{190, 30, 45}      
\definecolor{bgBlue}{RGB}{245, 248, 253}      
\definecolor{bgOrange}{RGB}{254, 250, 245}    

\theoremstyle{plain}
\newtheorem{theorem}{Theorem}[section]
\newtheorem{lemma}[theorem]{Lemma}
\newtheorem{proposition}[theorem]{Proposition}
\newtheorem{claim}[theorem]{Claim}
\newtheorem{corollary}[theorem]{Corollary}

\theoremstyle{definition}
\newtheorem{definition}[theorem]{Definition}
\newtheorem{problem}[theorem]{Problem}

\theoremstyle{remark}

\newcommand{\dist}{\operatorname{dist}}

\newcommand{\E}{\mathbb E}
\newcommand{\Prb}{\mathbb P}

\title{Induced subdivisions in graphs of large girth}
\author{Peiru Kuang\footnote{School of Mathematical Sciences, Shanghai Jiao Tong University, Shanghai 200240, China. Email: peiru\_k@sjtu.edu.cn} 
\and 
Yan Wang\footnote{School of Mathematical Sciences, Shanghai Jiao Tong University, Shanghai 200240, China. Supported by National Key R\&D Program of China under Grant No. 2022YFA1006400 and National Natural Science Foundation of China under Grant No. 12571376. Email: yan.w@sjtu.edu.cn (corresponding author).}}

\date{}

\begin{document}
\maketitle
\begin{abstract}
In this paper, we prove that there exists an absolute constant $g_0$ such that, for every integer $k\ge 3$, every graph $G$ with $\delta(G)\ge k$ and $g(G)\ge g_0$ contains an induced subdivision of $K_{k+1}$. This fully resolves a problem raised by K\"{u}hn and Osthus (originally attributed to Shi), and improves a recent result of Gir\~{a}o and Hunter. Our proof uses some ideas from Gir\~{a}o and Hunter. Another main ingredient in our proof is an induced variant of Mader's theorem: for every fixed \(s,\eta,D\), every graph \(J\) with \(\Delta(J)\le D\), \(d(J)>s-2+\eta\) and sufficiently large girth contains an induced subdivision of \(K_s\).
\end{abstract}

\section{Introduction}
A subdivision of a graph \(H\) is a graph obtained from \(H\) by replacing each edge by an internally vertex-disjoint path. Subdivisions have played an important role in the connections between graph theory and topology; for example, Kuratowski's theorem \cite{Kuratowski1930} characterizes planar graphs by excluding subdivisions of $K_5$ or $K_{3,3}$.

A fundamental extremal question asks how many edges are needed to force a subdivision of a complete graph. In 1967, Mader \cite{Mader1967} proved that, for every integer \(r\), there is a positive number \(d(r)\) such that every graph of average degree at least \(d(r)\) contains a subdivision of \(K_r\). The correct order of \(d(r)\) is quadratic: complete bipartite graphs give the lower bound (first observed by Jung \cite{Jung1970}), while Bollob\'as and Thomason \cite{BollobasThomason1998} and, independently, Koml\'os and Szemer\'edi \cite{KomlosSzemeredi1994,KomlosSzemeredi1996} proved the matching upper bound up to a constant. 

The induced analogue of this problem is rather different. Finding induced subdivisions is inherently more difficult than ordinary ones. We must ensure that the paths are internally vertex-disjoint, and also forbid any additional edges between vertices that are not adjacent in the subdivision. Indeed, K\"uhn and Osthus \cite{KuhnOsthusInduced2004} observed that Mader's theorem fails for induced subdivisions unless dense bipartite subgraphs are excluded. Recent work of Gir\~ao and Hunter~\cite{GiraoHunter2025IMRN} also shows that,
in \(K_{s,s}\)-free graphs, polynomial average degree forces an induced
subdivision of any fixed graph. To avoid such dense subgraphs, a natural approach is to impose some local sparsity conditions, such as a large girth condition. 

The existence of finite $r$-regular graphs of arbitrarily large girth was proved by Sachs~\cite{Sachs1963}. Remarkably, his construction yields graphs that exhibit highly structured global properties, such as being Hamiltonian and having a 2-factor. Despite this local sparsity, a minimum-degree condition imposes substantial global structure. Indeed, for every \(k\), Thomassen~\cite{Thomassen1983} proved that every graph with minimum degree at least \(3\) and sufficiently large girth contains a \(K_k\)-minor. For subdivisions, a classical result of Mader \cite{Mader1967} implies that a sufficiently large minimum degree (dependent only on $r$) forces a subdivision of $K_{r+1}$. Later, Mader \cite{Mader1998} showed that $\delta(G)\ge r$ suffices.

\begin{theorem}[Mader \cite{Mader1998}]\label{Mader1}
For every $r$, there is a function $g(r)$ such that every graph with $\delta(G)\ge r$ and girth at least $g(r)$ contains a subdivision of $K_{r+1}$.
\end{theorem}
This is best possible in terms of the minimum degree. Mader \cite{Mader1998} showed that the function $g(r)$ in Theorem \ref{Mader1} can be chosen linear in \(r\). He \cite{Mader1999} later asked whether girth at least five might already suffice for all \(r\). Subsequent work of K\"uhn and Osthus showed that the girth threshold can indeed be bounded by an absolute constant independent of \(r\): in \cite{KuhnOsthus2002}, they initially established a general girth bound of $186$, which was reduced to $15$ for sufficiently large $r$; and subsequently in \cite{KuhnOsthus2006}, they improved the girth bound to $27$ for all $r$ (see Theorem 2.11 of \cite{KuhnOsthus2006}). In the same paper \cite{KuhnOsthus2006}, they asked, in a problem originally attributed to Shi, whether an analogous statement holds for induced subdivisions. 

\begin{problem}[K\"uhn and Osthus \cite{KuhnOsthus2006}]\label{p1}
Is there a function \(h(r)\) such that every graph of minimum degree at least \(r\ge 3\) and girth at least \(h(r)\) contains an induced subdivision of
\(K_{r+1}\)?
\end{problem}

In this paper, we answer this problem affirmatively in a strong sense: $h(r)$ can be chosen as an absolute constant.

\begin{theorem}\label{thm:main}
There exists an absolute constant \(g_0\) such that for every integer \(k\ge3\), every graph \(G\) with $\delta(G)\ge k$ and $g(G)\ge g_0$ contains an induced subdivision of \(K_{k+1}\).
\end{theorem}

We have made no attempt to optimize the value of \(g_0\), as doing so would make the paper difficult to read. In fact, our proof gives \(g_0=8\cdot 10^6\). Thus Theorem \ref{thm:main} improves a result of Gir\~ao and Hunter \cite{GiraoHunter2026}, where they require $k\geq 10^8$ and $g_0 \geq 10^8$. Further discussion on the value of \(g_0\) can be found in the final section. 

By Theorem \ref{thm:main} (also Theorem \ref{Mader1}), every graph with average degree larger than \(2s\) and sufficiently large girth contains a subdivision of \(K_{s+2}\). Mader conjectured in \cite{Mader1999} that the threshold $2s$ could be lowered to \(s+1\). This was later confirmed, in a stronger asymptotic form, by Mader himself~\cite{Mader2001}.
\begin{theorem}[Mader \cite{Mader2001}]\label{ave-thm}
For every integer $s\ge4$ and every $\eta>0$, every graph with average degree at least \(s-2+\eta\) and sufficiently large girth contains a subdivision of \(K_{s}\).
\end{theorem}
Theorem \ref{ave-thm} does not remain true for $\eta=0$, since there are $(s-2)$-regular graphs of arbitrarily large girth (so they contain no subdivisions of $K_s$) as shown by Sachs~\cite{Sachs1963}.
A key ingredient in the proof of Theorem~\ref{thm:main} is the following induced analogue of Theorem \ref{ave-thm}, which may be of independent interest. It shows that, for graphs of bounded maximum degree, sufficiently large girth together with average degree slightly larger than \(s-2\) forces an induced subdivision of \(K_{s}\), matching the size of complete graphs in Mader's theorem.
\begin{theorem}\label{induced Mader}
For every integer $s\ge4$, every $\eta>0$, and every integer $D\ge s-1$, there is an integer $g_{s,\eta,D}$ such that every graph $J$ satisfying
$\Delta(J)\le D$, $d(J)>s-2+\eta$ and $g(J)\ge g_{s,\eta,D}$ contains an induced subdivision of $K_s$.
\end{theorem}
The bounded maximum degree assumption in Theorem~\ref{induced Mader} is essential in our proof. Its proof combines a robust-expansion argument, in the spirit of the Koml\'os--Szemer\'edi sublinear-expander framework \cite{KomlosSzemeredi1994,KomlosSzemeredi1996} and its recent applications to the subdivisions of complete graphs \cite{BucicMontgomery2024,LiuMontgomery2023,Wang2023}, with a probabilistic method. We expect this may be useful in other problems on induced subdivisions. 
\subsection{Applications}
A graph $G$ is said to be \textit{$\chi$-bounded} with a $\chi$-bounding function $f$ if, for all induced subgraphs $G'$ of $G$, we have $\chi(G') \leq f(\omega(G'))$. A class of graphs is $\chi$-bounded if there exists a $\chi$-bounding function that holds for all graphs of the class, see Gy\'arf\'as \cite{Gyarfas1987}. One of the most celebrated results is the Strong Perfect Graph Theorem by Chudnovsky, Robertson, Seymour and Thomas \cite{ChudnovskyRobertsonSeymourThomas2006}. The study of graphs with prescribed forbidden induced subgraphs, such as holes or subdivisions, is a major topic in structural graph theory. This is also closely related to the Erd\H{o}s-Hajnal conjecture (see, e.g., \cite{ChudnovskyScottSeymourSpirkl2023}). Scott \cite{Scott1997} conjectured that, for every fixed graph $H$, the class of graphs containing no induced subdivision of $H$ is $\chi$-bounded. A particularly natural case is when $H$ is a complete graph. In this direction, L\'ev\^eque, Maffray, and Trotignon~\cite{LevequeMaffrayTrotignon2012}, and Trotignon and Vu\v{s}kovi\'c~\cite{TrotignonVuskovic2017} proved that every graph of girth at least \(5\) with no induced subdivision of \(K_4\) is \(3\)-colorable (see Theorem 2.12 of \cite{TrotignonVuskovic2017}). As an immediate consequence of Theorem~\ref{thm:main}, we obtain the corresponding result for all complete graphs.
\begin{corollary}
Let \(g_0\) be as in Theorem~\ref{thm:main}. For every integer \(r\ge4\), every graph of girth at least \(g_0\) with no induced subdivision of \(K_r\) is \((r-2)\)-degenerate. In particular, it is \((r-1)\)-colorable.
\end{corollary}

A well-known conjecture of Haj\'os (see, e.g., \cite{JensenToft1995}) states that every graph of chromatic number at least $r$ contains a subdivision of $K_r$. This conjecture is true for $r \le 4$~\cite{Dirac1952}. However, it was disproved by Catlin \cite{Catlin1979} when $r\geq 7$, and Erd\H{o}s and Fajtlowicz \cite{ErdosFajtlowicz1981} proved that it fails even for almost all graphs. The relation between chromatic number and clique subdivisions has remained a central theme; see, for example, the work of Fox, Lee and Sudakov~\cite{FoxLeeSudakov2013}. On the other hand, since every graph of chromatic number at least $r$ has a subgraph of minimum degree at least $r-1$, Theorem \ref{thm:main} shows that the conjecture does hold for all graphs of sufficiently large girth, even in induced form.
\begin{corollary}
Let \(g_0\) be as in Theorem~\ref{thm:main}. Then every graph of chromatic number at least $r\geq 3$ and girth at least $g_0$ contains an induced subdivision of $K_r$, and thus a subdivision of $K_r$.
\end{corollary}
\medskip
The paper is organized as follows. In Section~\ref{sec:prelim}, we introduce the necessary definitions and some tools that we will need later. Section~\ref{sec:cleaning} establishes structures of graphs without induced subdivisions of complete graphs. Section~\ref{sec:robust} is devoted to proving the induced version of Mader's result, i.e., Theorem \ref{induced Mader}. In Section~\ref{sec:main-proof}, we prove our main result, Theorem~\ref{thm:main}. In the final section, we give some concluding remarks.

\paragraph{Notation.}
For a graph \(G\), we write \(V(G)\) and \(E(G)\) for its vertex set and edge set, respectively. We also write \(|G|:=|V(G)|\). We denote by \(d(G)\), \(\delta(G)\), \(\Delta(G)\), and \(g(G)\) the average degree, minimum degree, maximum degree, and girth of \(G\), respectively. For a vertex \(v\in V(G)\), we write \(N_G(v)\) for its neighborhood and \(d_G(v):=|N_G(v)|\) for its degree. For \(S\subseteq V(G)\), we write
\(
    N_G(S):=\{v\in V(G)\setminus S: v \text{ has a neighbor in } S\}.
\)
For \(S\subseteq V(G)\) and \(v\in V(G)\), we write
\(
    d_S(v):=|N_G(v)\cap S|.
\)
For \(U\subseteq V(G)\), we write \(G[U]\) for the subgraph of \(G\) induced by \(U\), and \(G-U\) for \(G[V(G)\setminus U]\). 

For disjoint sets \(X,Y\subseteq V(G)\), we write \(e_G(X,Y)\) for the number of edges with one endpoint in \(X\) and the other in \(Y\). For \(X\subseteq V(G)\), we write \(e_G(X):=|E(G[X])|\). We write \(\dist_G(u,v)\) for the distance between vertices \(u,v\in V(G)\), and \(c_G(X)\) for the number of connected components of \(G[X]\). A path with endpoints \(u\) and \(v\) is called a \(u,v\)-path and may be denoted by \(P_{uv}\). When the underlying graph is clear, we omit the subscript \(G\). All logarithms are natural.

\section{Preliminaries}\label{sec:prelim}
We begin with the notion of degeneracy, which will be used to describe both sparsity and coloring.
\begin{definition}
Let \(d\) be a nonnegative integer. A graph \(G\) is \emph{\(d\)-degenerate} if every nonempty induced subgraph of \(G\) contains a vertex of degree at most \(d\).
\end{definition}
The next proposition presents two elementary facts of degeneracy that we use repeatedly. The proofs are easy and omitted.

\begin{proposition}\label{prop:degenerate-basic}
Let \(d\) be a nonnegative integer, and let \(G\) be a \(d\)-degenerate graph. Then the following hold.
\begin{enumerate}[label=\textup{(\roman*)}]
    \item Every induced subgraph \(H\) of \(G\) satisfies $e(H)\le d|V(H)|$.
    \item \(G\) is \((d+1)\)-colorable.
\end{enumerate}
\end{proposition}
We also require the following lemma, which allows us to pass to an induced subgraph with large minimum degree in any graph with large average degree.
\begin{lemma}\label{lem:avg-core}
Let \(d\ge 0\) be an integer. If \(G\) is a graph with \(d(G)>2d\),
then \(G\) contains a nonempty induced subgraph \(H\) with \(\delta(H)\ge d+1\).
\end{lemma}

\begin{proof}
Suppose not. Then every nonempty induced subgraph of \(G\) contains a vertex of
degree at most \(d\). In other words, \(G\) is \(d\)-degenerate. By
Proposition~\ref{prop:degenerate-basic}, we have
\(
e(G)\le d|V(G)|.
\)
Hence
\(
d(G)=\frac{2e(G)}{|V(G)|}\le 2d
\), a contradiction.
\end{proof}
We also need the following two probabilistic tools.
\begin{lemma}[Chernoff bound \cite{AlonSpencer}]\label{lem:chernoff}
Let \(X=\sum_{i=1}^n X_i\), where \(X_1,\dots,X_n\) are independent random variables with values in \([0,1]\), and let \(\mu=\mathbb E[X]\). Then, for every \(0<\delta<1\),
\[
\mathbb P\!\left(X\le (1-\delta)\mu\right)\le \exp\!\left(-\frac{\delta^2\mu}{2}\right).
\]
\end{lemma}

\begin{lemma}[Lov\'asz Local Lemma {\cite{AlonSpencer,ErdosLovasz1975}}]\label{LLL}
Let $(\Omega,\mathbb{P})$ be a probability space and let $A_1,\dots,A_t$ be events with $\mathbb{P}(A_i)\le p$ for all $i$. Suppose that each $A_i$ is independent of all but at most $D$ of the other events. If
\(
p\le \frac{1}{e(D+1)},
\)
then
\[
\mathbb{P}\!\left(\bigcap_{i=1}^t A_i^c\right)\geq (1-ep)^t .
\]
\end{lemma}
To bound the order of the graph, we rely on the following Moore-type bound under a large-girth assumption.
\begin{lemma}[Alon, Hoory and Linial {\cite{AlonHooryLinial2002}}]\label{lem:moore}
Let $m\ge 0$, and let $G$ be a graph with $\delta(G)\geq \delta\ge 2$ and 
$g(G)\geq 2m+2$. Then
\(
|V(G)|\ge 2\sum_{i=0}^m (\delta-1)^i.
\)
\end{lemma}
Another important definition we need is $k$-linkage. It allows one to find a subdivision of $K_{\Omega(k^{1/2})}$ in a robust way, see for example, \cite{BollobasThomason1998}.
\begin{definition}
Given a graph $G$ on at least $2k$ vertices, we say $G$ is \emph{$k$-linked} if for any two disjoint sets of $k$ vertices
$\{x_1, \ldots,x_k\}$, $\{y_1,\ldots, y_k\}$, we can find vertex-disjoint paths $P_1, \ldots, P_k$ where the end vertices of $P_i$ are $x_i$ and $y_i$.
\end{definition}
We will make essential use of the following result of Thomas and Wollan \cite{ThomasWollan2005}.
\begin{theorem}[Thomas and Wollan {\cite{ThomasWollan2005}}]\label{lem:TW}
If $G$ is $10k$-connected, then $G$ is $k$-linked.
\end{theorem}
We also use the following theorem, which yields a highly connected induced subgraph with controlled boundary.
\begin{theorem}[Penev, Thomass\'e and Trotignon \cite{PTT2016}]\label{PTT1}
Let \(k\) be a positive integer, and let \(G\) be a graph. If \(\delta(G)>2k^2-1\),
then \(G\) contains a \((k+1)\)-connected induced subgraph \(H\) such that
\(\partial(H)\subsetneq V(H)\) and \(|\partial(H)|\le 2k^2-1\), where 
\(
\partial(S):=\{v\in S: N(v)\setminus S\neq \emptyset\}
\)
is the boundary of \(S\).
\end{theorem}
Applying Theorem \ref{PTT1} with \(k=q-1\), we immediately obtain the following corollary.
\begin{corollary}\label{PTT}
Let \(q\ge 2\), and let \(G\) be a graph with minimum degree at least \(4q^2\).
Then \(G\) contains a \(q\)-connected induced subgraph \(H\) with more than \(4q^2\)
vertices such that
\[
\bigl|\{v\in V(H):N_G(v)\setminus V(H)\neq \emptyset\}\bigr|\le 2q^2.
\]
\end{corollary}
Finally, we need the following result to find subdivisions (not necessarily induced) in graphs of moderate girth.
\begin{theorem}[K\"uhn and Osthus {\cite{KuhnOsthus2006}}] \label{KO-topological-clique}
Let $r\ge 1$ be an integer. Every graph of minimum degree at least $r$ and girth at least $27$ contains a subdivision of $K_{r+1}$.
\end{theorem}
The study of graphs with no induced subdivision of \(K_4\) was initiated in \cite{LevequeMaffrayTrotignon2012}. The case of induced subdivisions of \(K_4\) will be handled by the following theorem.
\begin{theorem}[Trotignon and Vu\v{s}kovi\'c {\cite{TrotignonVuskovic2017}}]\label{thm:TV-ISK4}
Every graph of girth at least \(5\) that does not contain an induced subdivision of $K_4$ has a vertex of degree at most \(2\).
\end{theorem}

\section{Structures of graphs without induced subdivisions of complete graphs}\label{sec:cleaning}
The first lemma allows us to find an induced subdivision of a complete graph in a very unbalanced bipartite graph.
\begin{lemma}\label{lem:unbalanced}
Let $d\ge2$, and let $G$ be a $d$-degenerate graph with $g(G)\ge54$. Suppose that $G$ has disjoint vertex sets $A,B$ such that $|A|> 60d^2|B|$ and $d_B(a)\ge2$ for every $a\in A$. Then $G$ contains an induced subdivision of $K_{d+1}$.
\end{lemma}

\begin{proof}
Let
\(
  A_0:=\{a\in A:d_B(a)\le4d\}.
\)
For each $a\in A\setminus A_0$ we have $d_B(a)>4d$, and thus
\(
4d\,|A\setminus A_0|<e(A,B)\le e(G[A\cup B])\le d(|A|+|B|).
\)
Therefore
\(
|A\setminus A_0|<(|A|+|B|)/{4}\le 3|A|/8,
\)
as $|A|\ge 60d^2|B|\ge 2|B|$. Hence $|A_0|>|A|-3|A|/8=5|A|/8$.

For each $a\in A_0$, fix two distinct vertices $u_a,v_a\in N_B(a)$ and let $\pi(a):=\{u_a,v_a\}$. Since $g(G)\ge 54$, $u_av_a\notin E(G)$. Now choose each vertex of $B$ independently at random with probability $p:=1/(6d)$. Fix a \(d\)-degenerate ordering of \(G[B]\), so that each vertex has at most \(d\) right-neighbors. Let $R$ be the set of chosen vertices with no chosen right-neighbor. Then $R$ is an independent set. Call $a\in A_0$ \textit{good} if $N_R(a)=\pi(a)$. Since $d_B(a)\le4d$ and each of $u_a,v_a$ has at most $d$ right-neighbors,
\[
  \Prb(a\text{ is good})
  \ge p^2(1-p)^{(d_B(a)-2)+2d}
  \ge p^2(1-p)^{6d}
  >\frac1{108d^2}.
\]
Let $X$ be the number of good vertices. Then
\(
  \E X\ge |A_0|/(108d^2)\ge 5|A|/(864d^2).
\)

Let $F$ be the auxiliary graph on the good vertices in which two vertices are adjacent if they are adjacent in $G[A_0]$, and let $Y=e(F)$. If $aa'\in E(G[A_0])$, then $\pi(a)\cap\pi(a')=\varnothing$; otherwise, $G$ contains a triangle. Hence $\Prb(\text{both $a$ and $a'$ are good})\le p^4$, and
\(
  \E Y\le e(G[A_0])p^4\le d|A|\cdot (1/(6d))^4=|A|/({1296d^3}).
\)
Since $\E|R|\le p|B|=|B|/(6d)$ and $|A|\ge60d^2|B|$, we have
\[
  \E(X-Y-d|R|)
  \ge \frac{5|A|}{864d^2}-\frac{|A|}{1296d^3}-\frac{|B|}{6}>0.
\]
Thus, for some choice of $R$, we have $X-Y>d|R|$. Hence $F$ contains an independent set $I$ with $|I|\ge X-Y>d|R|$. Let $H$ be the graph with vertex set $R$ in which, for each $a\in I$, we add the edge $u_av_a$. The graph $H$ is simple, for if two distinct vertices $a,a'\in I$ correspond to the same edge $uv$, then $uava'u$ would form a 4-cycle in $G$. Thus $|E(H)|=|I|>d|R|$ and so $d(H)>2d$.

Note that $g(H)\ge 27$. Indeed, any cycle $r_1r_2\cdots r_t r_1$ in $H$ gives a cycle
\(
r_1a_1r_2a_2\cdots r_ta_t r_1
\)
of length $2t$ in $G$. By Lemma~\ref{lem:avg-core}, $H$ contains an induced subgraph $H'$ with $\delta(H')\ge d$. Since $g(H')\ge27$, by Theorem~\ref{KO-topological-clique}, we obtain a subdivision of $K_{d+1}$ in $H'$. Since $R$ is independent, $I$ is independent, and every vertex of $I$ has exactly two neighbors in $R$, the $1$-subdivision of every subgraph of $H$ is an induced subgraph of $G[I\cup R]$. It follows that $G$ contains an induced subdivision of $K_{d+1}$.
\end{proof}

The next lemma gives the cleaning step used in the case \(d\ge 4\).

\begin{lemma}\label{lem:one-sided}
Let \(d\ge4\), and let \(G\) be a \(d\)-degenerate \(n\)-vertex graph with \(g(G)\ge54\) and no induced subdivision of \(K_{d+1}\). Let \(X\subseteq V(G)\) satisfy \(|X|\ge 81n/100\) and \(d_G(x)\ge d\) for every
\(x\in X\). Suppose that there is a set \(B_0\subseteq V(G)\setminus X\) such that every vertex of \(X\) has a neighbor in \(B_0\). Then there exist disjoint sets \(X'\subseteq X\) and \(Y\subseteq V(G)\setminus X'\) such that:
\begin{enumerate}[label=\textup{(\roman*)}]
\item \(|Y|\ge 10^{-12}n/(d^3(d+1))\);
\item \(Y\) is independent;
\item \(2\le |N(y)\cap X'|\le 3\cdot10^7d^4\) for every \(y\in Y\);
\item \(d_G(x')\le 800d\) for every \(x'\in X'\).
\end{enumerate}
\end{lemma}

\begin{proof}
Let \(\Delta_0:=800d\). Since \(G\) is \(d\)-degenerate, the sum of the degrees of \(V(G)\) is at most \(2dn\). Hence, there are at most \(n/400\) vertices that have degree larger than \(\Delta_0\). Choose \(Z\subseteq X\) consisting of vertices of degree at most \(\Delta_0\) such that
\(|Z|\ge (81/100-1/400)n=323n/400\). Choose each vertex of \(Z\) independently at random with probability \(1/2\), and let \(Z_1\) be the resulting set. Then
\(
\mathbb E\, e(Z_1,V(G)\setminus Z_1)
=\frac{1}{2}e(Z) + \frac{1}{2}e(Z,V(G)\setminus Z)
=\frac14\sum_{z\in Z}d_G(z)+\frac14 e(Z,V(G)\setminus Z).
\)
Every vertex of \(Z\) has a neighbor in \(B_0\subseteq V(G)\setminus X\), and therefore \(e(Z,V(G)\setminus Z)\ge |Z|\). Since \(d_G(z)\ge d\) for every \(z\in Z\), $\mathbb E\, e(Z_1,V(G)\setminus Z_1) \ge (d+1)|Z|/4$. Thus we may fix a choice \(Z_1\) such that
\(
e(Z_1,V(G)\setminus Z_1)\ge (d+1)|Z|/{4}
\ge 323(d+1)n/{1600}. 
\)

Let \(\kappa:=3\cdot10^7d^4\),
\(U:=\{u\in V(G)\setminus Z_1:d_{Z_1}(u)\ge2\}\) and
\(U':=\{u\in U:d_{Z_1}(u)>\kappa\}\). Then, we have \(|U'|<2dn/\kappa\).

\begin{claim}\label{c3.3}
\(|U\setminus U'|>10^{-12}n/d^3\).
\end{claim}
Suppose not. Then \(e(Z_1,U\setminus U')\le \kappa |U\setminus U'|\le 3\cdot10^{-5}dn\). Each vertex in
\(V(G)\setminus (Z_1\cup U)\) has at most one neighbor in \(Z_1\), and hence,
\[
\begin{aligned}
e(Z_1,U')-|Z_1|
&=e(Z_1,U)-e(Z_1,U\setminus U')-|Z_1|  \\
&\ge e(Z_1,V(G)\setminus Z_1)-n+|U|-e(Z_1,U\setminus U')  \\
&\ge \left(\frac{323(d+1)}{1600}-1-3\cdot10^{-5}d\right)n
  >\frac{n}{200}.
\end{aligned}
\]
Delete from \(Z_1\) every vertex with fewer than two neighbors in \(U'\), and let \(Z_1^*\) be the remaining set. Then \(e(Z_1^*,U')\geq e(Z_1,U')-|Z_1|>n/200\). Since every vertex of \(Z_1^*\) has degree at most \(\Delta_0\), it follows that
\(
|Z_1^*|>n/(160000d).
\)
Moreover, \(|U'|<2n/(3\cdot10^7d^3)\), and therefore
\(
\frac{|Z_1^*|}{|U'|}
>
\frac{3\cdot10^7}{320000}d^2
>
60d^2.
\)
Every vertex of \(Z_1^*\) has at least two neighbors in \(U'\). Applying
Lemma~\ref{lem:unbalanced} with \(A:=Z_1^*\) and \(B:=U'\), we obtain an
induced subdivision of \(K_{d+1}\), a contradiction. This proves
Claim~\ref{c3.3}. \hfill$\blacksquare$

\medskip
Since \(G[U\setminus U']\) is \(d\)-degenerate, it is \((d+1)\)-colorable. Let \(Y\) be a largest color class. Then \(Y\) is independent and, by Claim~\ref{c3.3},
\(
|Y|\ge \frac{|U\setminus U'|}{d+1}>
\frac{10^{-12}n}{d^3(d+1)}.
\)
Taking \(X':=Z_1\) completes the proof.
\end{proof}
\section{Induced version of Mader's result}\label{sec:robust}
In this section, we prove an induced analogue of a classical result of Mader. We begin with a lemma that extracts a highly connected induced subgraph while preserving many vertices of a prescribed set with almost all of their degrees.
\begin{lemma}\label{lem:core}
Let $k\ge2$, $D$ and $m$ be positive integers. Let $G$ be an $n$-vertex graph satisfying $\delta(G)\ge9k^2$, $\Delta(G)\le D$ and $g(G)\ge2m+2$. Let $B\subseteq V(G)$ satisfy $|B|\ge n/D$. Assume that $2(k^2-1)^m\ge 11k^2D^2$. Then $G$ contains a $k$-connected induced subgraph $H$ with $|V(H)|\ge 11k^2D^2$ such that
\[
  \bigl|\{x\in B\cap V(H):d_H(x)\ge d_G(x)-2k^2\}\bigr|
  \ge \frac{|V(H)|}{2D}
  \ge \frac{11}{2}k^2D.
\]
\end{lemma}

\begin{proof}
Set $T:=4k^2$. We iteratively construct pairwise vertex-disjoint induced subgraphs $H_1,\ldots,H_\tau$. Let $G_1:=G$. For $t\geq 1$, if $G_t$ is non-empty, then by construction $\delta(G_t)\ge T$. By Corollary~\ref{PTT}, we can find a $k$-connected induced subgraph $H_t\subseteq G_t$ with $|V(H_t)|>4k^2$ and boundary
\(
S_t:=\{u\in V(H_t):N_{G_t}(u)\setminus V(H_t)\neq\varnothing\},
\)
where $|S_t|\le 2k^2$. Delete $V(H_t)$ and delete vertices with degree less than $T$ in the current graph until no more such vertices remain. Let $Z_t$ be the set of deleted vertices in the current step, and write
\(
  G_{t+1}:=G_t\setminus (V(H_t)\cup Z_t).
\)
The process stops when the remaining graph is empty. Let $S:=\bigcup_{t=1}^\tau S_t$ and $Z:=\bigcup_{t=1}^\tau Z_t$.

Orient each edge of $G[Z]$ toward the end vertex deleted later. Since every vertex of $Z$ had fewer than $T$ neighbors when it was deleted, every vertex has fewer than $4k^2$ out-neighbors in this orientation, and hence $e(G[Z])<4k^2|Z|$. Since \(d_G(v)\ge 9k^2\) for every \(v\in Z\), each vertex \(v\in Z\) has more than \(5k^2\) neighbors among previously deleted vertices. If such a neighbor lies in an earlier block $H_i$, then it must lie in $S_i$. Hence
\(
  e(Z,S)+e(G[Z])>5k^2|Z|.
\)
Thus, $e(Z,S)>k^2|Z|$. Since $\Delta(G)\le D$, we have $|Z|\le D|S|/{k^2}$.

For each $t$, write $n_t:=|V(H_t)|$. If $u\in V(H_t)\setminus S_t$, then $d_{H_t}(u)=d_{G_t}(u)\ge4k^2$. If $u\in S_t$, then $d_{H_t}(u)\ge k$ as $H_t$ is $k$-connected. Since \( |S_t|\le 2k^2<n_t/2 \), we have $2e(H_t)\ge (n_t-|S_t|)4k^2+|S_t|k  =4k^2n_t-|S_t|(4k^2-k) >4k^2n_t-n_t(4k^2-k)/{2} 
>2k^2 n_t$. Thus \(d(H_t)>2k^2\). By Lemma \ref{lem:avg-core}, $H_t$ contains an induced subgraph $J_t$ with $\delta(J_t)\ge k^2$. By Lemma~\ref{lem:moore} and our assumption,
\(
  |V(H_t)|\ge |V(J_t)|\ge 2\sum_{i=0}^m (k^2-1)^i\ge2(k^2-1)^m\ge 11k^2D^2.
\)
Thus $\tau\le n/(11k^2D^2)$, and $|S|\le2k^2\tau\le 2n/(11D^2)$. Also, we have $e(Z,V(G)\setminus(S\cup Z))<4k^2|Z|$.
Let
\(
  Q:=\{v\in V(G)\setminus(S\cup Z):d_Z(v)>2k^2\}.
\)
Hence, $|Q|<2|Z|$. Let
  $W:=S\cup Z\cup N_G(S)\cup Q$
and
  $B':=B\setminus W$.
Thus, 
\(
  |W|\le |S|+|Z|+D|S|+|Q|
  <(D+1)|S|+3|Z|
  \le\left(D+1+3D/{k^2}\right)|S|
  \le2D|S|<n/{2D}.
\)
Since $|B|\ge n/D$, it follows that $|B'|\ge n/(2D)$.

Note that $H_1,\ldots,H_\tau$ partition $V(G)\setminus Z$ and $B'\cap Z=\varnothing$. By averaging, we have, for some $t$,
\(
  \frac{|B'\cap V(H_t)|}{|V(H_t)|}\ge \frac{|B'|}{n}\ge\frac1{2D}.
\)
Let $H := H_t$, satisfying $|V(H)| \ge 11k^2D^2$. For any $x \in B' \cap V(H)$, $x \notin W$ implies $x \notin S_t \cup N_G(S)$. Consequently, $x$ cannot have neighbors in $G_t \setminus H$ and no neighbor in any earlier block. Thus, all neighbors of $x$ outside $H$ must lie in $Z$. Since $x \notin Q$, we obtain $d_H(x) \ge d_G(x) - 2k^2$. This completes the proof.
\end{proof}
We also need the following definition.
\begin{definition}\label{def:robust}
Let \(G\) be a graph, and let \(H\) be an auxiliary graph with $V(H)\subseteq V(G)$. Suppose that for each edge \(uv\in E(H)\), an induced path \(P_{uv}\) in \(G\) is fixed. Let \(a,q\ge 1\) be integers. A vertex \(y\in V(H)\) is called \emph{\(q\)-robustly \(a\)-branchable} if there exist distinct vertices
\(
  z_1(y),\ldots,z_a(y)\in N_G(y)
\)
and pairwise disjoint sets
\(
  M_1(y),\ldots,M_a(y)\subseteq N_H(y)
\)
such that
\begin{enumerate}[label=\textup{(\roman*)}]
  \item $|M_i(y)|\ge q$ for every $i\in[a]$;
  \item $z_i(y)\in V(P_{yy'})$ for every $i\in[a]$ and every $y'\in M_i(y)$;
  \item for every choice of $y_i\in M_i(y)$, $i\in[a]$, the paths
  \(
    P_{yy_1},\ldots,P_{yy_a}
  \)
  are pairwise internally vertex-disjoint and pairwise anticomplete in $G-y$.
\end{enumerate}
\end{definition}
Now we prove Theorem \ref{induced Mader}, which is an induced analogue of a classical result of Mader. For convenience, we restate it here.
\begin{theorem}\label{thm:avg-mader}
For every integer $s\ge4$, every $\eta>0$, and every integer $D\ge s-1$, there is an integer $g_{s,\eta,D}$ such that every graph $J$ satisfying
$\Delta(J)\le D$, $d(J)>s-2+\eta$ and $g(J)\ge g_{s,\eta,D}$
contains an induced subdivision of $K_s$.
\end{theorem}

\begin{proof}
Let $a:=s-1$, $\alpha:=(s-2)/2+\eta/4$ and $\gamma:=\alpha-1$. Let $q:=10a^2+a+1$ and $Q:=9q^2$. Let $L:=4\ell+1$, $C:=(L+1)(D+1)$, $p:=1/(C+1)$, $c_0:=\eta/(4(D-s+2)D^{2\ell})$, $\mu_\ell:=\gamma\alpha^{\ell-3}/(e(C+1))$ and $\pi:=\exp(-\mu_\ell/8)$. Choose $\ell$ sufficiently large such that $\mu_\ell\ge 18Q$, $\pi\le 1/({e(4D^{12\ell+5}+1)})$ and $2eD\pi<pc_0/{8}$.
This is possible since $\alpha>1$. Define
\(
D_0:=\left\lceil
\max\left\{2/{pc_0},\,2D^{2\ell+1},\,Q,\,D\right\}
\right\rceil .
\) Choose $m$ sufficiently large so that $2(q^2-1)^m \ge 11q^2 D_0^2$.
Let $g_{s,\eta,D}:=\max\{12\ell+5,L(2m+2)\}$.

Suppose that $J_0$ is a counterexample. Since
$d(J_0)>s-2+\eta$, we have $e(J_0)>\alpha |J_0|$. Passing to a vertex-minimal induced subgraph \(J\subseteq J_0\) with \(e(J)>\alpha |J|\), we may assume that \(J\) is minimal with this property. Thus, $e(J-X)\le \alpha(|J|-|X|)$ for every non-empty proper set $X\subseteq V(J)$. Taking $X=\{v\}$, we have $d_J(v)>\alpha>1$, and hence $\delta(J)\ge 2$. Let $U:=\{v\in V(J):d_J(v)\ge a\}$. Note that $\sum_{v\in V(J)}(d_J(v)-(s-2))=2e(J)-(s-2)|J|>\eta |J|/2$. Since only vertices of $U$ contribute positively to this sum, and each contributes at
most $D-(s-2)$, we have $|U|\ge \eta|J|/({2(D-s+2)})$. Also note that $g(J) \ge g(J_0) \ge g_{s,\eta,D}$.

For every non-empty proper set \(X\subseteq V(J)\) such that \(J[X]\) is a forest with \(c(X)\) components, we have $e(J[X])=|X|-c(X)$. Thus, we have $e(X,V(J)\setminus X)=e(J)-e(J-X)-e(J[X])>\alpha|J|-\alpha(|J|-|X|)-|X|+c(X) =\gamma |X|+c(X)$.
Fix $y\in V(J)$ and $z\in N_J(y)$. For $i\ge1$, let $X_i=X_i(y,z)$ be the vertex set of the component of
$J[\{v:1\le \operatorname{dist}_J(y,v)\le i\}]$ containing $z$. If $i\le \ell-2$, then $J[X_i]$ is a tree. Since \(zy\in E(J)\) and \(y\notin X_i\), one edge joining \(X_i\) to \(V(J)\setminus X_i\) is \(zy\). Every other such edge has one of its end vertices in \(X_{i+1}\setminus X_i\), and these end vertices are pairwise distinct by the girth assumption. Therefore \(|X_{i+1}\setminus X_i|>\gamma |X_i|\), and so \(|X_{i+1}|>\alpha |X_i|\). Write \(L_z(y):=X_{\ell-1}\setminus X_{\ell-2}\). Thus, $|L_z(y)|>\gamma \alpha^{\ell-3}$.

Take a maximal set of vertices \(U'\subseteq U\) with pairwise distance greater than \(2\ell\). Since \(\Delta(J)\le D\), every ball of radius \(2\ell\) has fewer than \(2D^{2\ell}\) vertices. Thus, $|U'|\geq |U|/(2D^{2\ell})\geq c_0|J|$. Extend \(U'\) to a maximal set \(S^*\subseteq V(J)\) whose vertices have pairwise distance greater than \(2\ell\). Fix a nearest-root breadth-first forest rooted at \(S^*\). For each \(w\in S^*\), let \(T_w\) denote the tree rooted at \(w\). Since each \(J[T_w]\) is an induced tree and there is at most one edge between two distinct trees, for every pair \(u,v\in S^*\) with an edge between \(T_u\) and \(T_v\), there is a unique \(u\)-\(v\) path in \(J[T_u\cup T_v]\). Fix this path as \(P_{uv}\).

Define $H^*$ on vertex set $S^*$ by joining $u,v$ whenever there is a path $P_{uv}$ of length at most $L$ contained in $T_u\cup T_v$. Define a bipartite graph $F$ with parts $E(H^*)$ and $S^*$ by joining $e\in E(H^*)$ to $w\in S^*$ if $P_e$ meets $T_w$ or some vertex of $P_e$ has a neighbor in $T_w$. Since $|V(P_e)|\le L+1$, we have $|N_F(e)|\le C$ for every $e\in E(H^*)$.

Fix \(y\in S^*\) and \(z\in N_J(y)\). First note that \(L_z(y)\subseteq T_y\). For each \(\lambda\in L_z(y)\), choose a leaf \(x_\lambda\) of \(J[T_y]\) such that \(\lambda\in P_{yx_\lambda}\). Since \(\delta(J)\ge2\), \(x_\lambda\) has a neighbor \(r_\lambda\notin T_y\). Let \(w_\lambda\in S^*\) be the root such that \(r_\lambda\in T_{w_\lambda}\). The path from \(y\) to \(x_\lambda\), followed by the edge \(x_\lambda r_\lambda\) and the tree path in \(T_{w_\lambda}\) from \(r_\lambda\) to \(w_\lambda\), has length at most \(4\ell+1=L\). Thus \(e_\lambda:=yw_\lambda\in E(H^*)\). Moreover, since \(\lambda\in L_z(y)\) and \(x_\lambda\) is chosen in the descendant
subtree of \(\lambda\), the path \(P_{e_\lambda}\) starts at \(y\) and
\(z\in V(P_{e_\lambda})\). The construction is illustrated in Figure~\ref{fig:jctb_style_subdivision}.

\begin{figure}[htbp]
    \centering
    \begin{tikzpicture}[
        scale=0.8,
        vtx/.style={circle, draw=black, fill=black, inner sep=0pt, minimum size=4.5pt, thick},
        tree_y_style/.style={draw=jctbBlue!70, thick, top color=jctbBlue!15, bottom color=bgBlue, rounded corners=6mm},
        tree_w_style/.style={draw=jctbOrange!70, thick, top color=jctbOrange!15, bottom color=bgOrange, rounded corners=6mm},
        phys_edge/.style={draw=pathBlack, line width=1.2pt},
        phys_path/.style={draw=pathBlack, line width=1.2pt, dashed, dash pattern=on 4pt off 2pt},
        logic_edge/.style={draw=logicRed, line width=1.3pt, dash pattern=on 5pt off 3pt, -{Stealth[length=3.5mm]}},
        level_guide/.style={draw=black!25, thin, dash pattern=on 2pt off 2pt}
    ]
        \begin{scope}

            \draw[tree_y_style] (0, 4.5) -- (-2.6, 0) -- (2.6, 0) -- cycle;
            \node[text=jctbBlue!90!black, font=\large\bfseries] at (-1.8, 0.5) {$T_y$};

            \node[vtx, label=above:{$y \in S^*$}] (y) at (0, 4.5) {};
            \node[vtx, label=left:{$z$}] (z) at (0, 3.6) {};
            \node[vtx, label=left:{$\lambda$}] (lam) at (0.8, 1.8) {};
            \node[vtx, label=below:{$x_\lambda$ (leaf)}] (xlam) at (1.5, 0) {};

            \draw[level_guide] (-1.6, 1.8) -- (1.6, 1.8) node[right, text=black!50, font=\small] {$L_z(y)$};

            \draw[phys_edge] (y) -- (z);
            \draw[phys_path] (z) .. controls (0, 2.8) and (0.5, 2.2) .. (lam);
            \draw[phys_path] (lam) .. controls (1.1, 1.2) and (1.3, 0.6) .. (xlam);
        \end{scope}

        \begin{scope}[xshift=6.5cm]

            \draw[tree_w_style] (0, 4.5) -- (-2.6, 0) -- (2.6, 0) -- cycle;
            \node[text=jctbOrange!90!black, font=\large\bfseries] at (1.8, 0.5) {$T_{w_\lambda}$};

            \node[vtx, label=above:{$w_\lambda \in S^*$}] (wlam) at (0, 4.5) {};
            \node[vtx, label={[xshift=1.2cm, yshift=-0.5cm]$r_\lambda \notin T_y$}] (rlam) at (-1.8, 0.8) {};

            \draw[phys_path] (rlam) .. controls (-1.2, 2) and (-0.5, 3) .. (wlam);
        \end{scope}

        \draw[phys_edge] (xlam) -- (rlam) 
            node[midway, below, yshift=-4pt, text=black!70, font=\footnotesize]{};

        \begin{pgfonlayer}{background}
            \draw[logic_edge] (y) to[bend left=35] 
                node[above, midway, text=logicRed!80!black, font=\bfseries, yshift=2pt] 
                {$e_\lambda := yw_\lambda \in E(H^*)$} 
                (wlam);
        \end{pgfonlayer}

    \end{tikzpicture}
    
    \vspace{0.4cm}
    \caption{\small \textit{Local construction of the path \(P_{e_\lambda}\) and the corresponding edge \(e_\lambda\) in the auxiliary graph \(H^*\).}}
    \label{fig:jctb_style_subdivision}
\end{figure}
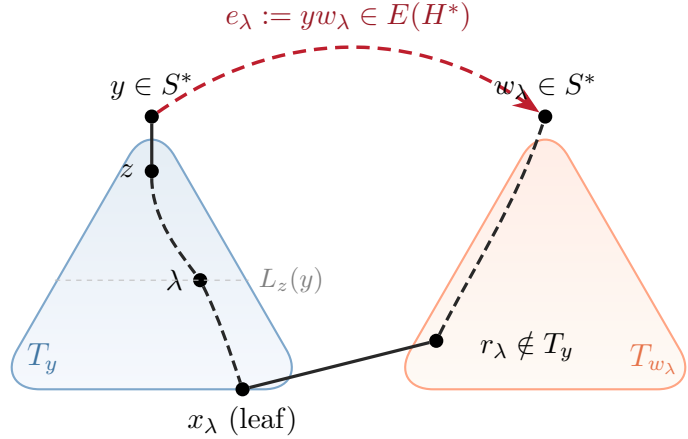

\begin{claim}\label{c4.3}
Let \(y\in S^*\), \(z\in N_J(y)\) and let \(\lambda,\lambda'\in L_z(y)\) be distinct. Then $N_F(e_\lambda)\cap N_F(e_{\lambda'})=\{y\}$.
\end{claim}

Suppose that \(w\in N_F(e_\lambda)\cap N_F(e_{\lambda'})\setminus\{y\}\). For each \(\xi\in\{\lambda,\lambda'\}\), there exist a vertex \(a_\xi\in V(P_{e_\xi})\) and an \(a_\xi\)-\(w\) path \(Q_\xi\) of length at most \(2\ell+1\) whose internal vertices lie in \(T_w\). Let \(u\) be the last common vertex of the \(y\)-\(\lambda\) and \(y\)-\(\lambda'\) paths in the tree \(T_y\). Since \(\lambda,\lambda'\in L_z(y)\) are distinct, we have \(\dist_J(y,u)\le\ell-2\). $y P_{e_\lambda} a_\lambda Q_\lambda w$ and $y P_{e_\lambda'} a_{\lambda'} Q_{\lambda'} w$ are therefore two distinct \(y\)-\(w\) walks. Each has length at most \(L+(2\ell+1)=6\ell+2\), so their union contains a cycle of length at most \(12\ell+4\), a contradiction to the girth assumption. This proves Claim \ref{c4.3}. \hfill$\blacksquare$

\medskip
Choose each vertex of $S^*$ independently with probability $p$, and let $S$ be the chosen set. Let \(H\) be the spanning subgraph of \(H^*[S]\) whose edge set consists of all \(uv\in E(H^*)\) such that \(N_F(uv)\cap S=\{u,v\}\). If \(e,f\in E(H)\) are vertex-disjoint, then \(P_e\) and \(P_f\) are vertex-disjoint and anticomplete in \(J\). For \(y\in S^*\) and \(z\in N_J(y)\), let \(E_{y,z}\) be the event that \(y\in S\) and fewer than \(Q\) edges \(yy'\in E(H)\) satisfy \(z\in V(P_{yy'})\).
\begin{claim} \label{c4.4}
$\mathbb P(E_{y,z})\le \exp(-\mu_\ell/8)=\pi$.
\end{claim}

For \(\lambda\in L_z(y)\), write \(I_\lambda:=\mathbf 1_{\{e_\lambda\in E(H)\}}\) and \(Z_{y,z}:=\sum_{\lambda\in L_z(y)} I_\lambda\). It suffices to show that $\mathbb P(Z_{y,z}<Q\mid y\in S) \le \pi$ since, conditional on \(y\in S\), the event \(E_{y,z}\) is contained in \(\{Z_{y,z}<Q\}\). For \(\lambda\in L_z(y)\), the event \(\{I_\lambda=1\}\) is determined by the vertices of \(N_F(e_\lambda)\setminus\{y\}\) that lie in \(S\). By Claim \ref{c4.3}, the random variables \(I_\lambda\), \(\lambda\in L_z(y)\), are mutually independent conditional on \(y\in S\). Moreover, if \(e_\lambda=yw_\lambda\), then
\[
\mathbb P(I_\lambda=1\mid y\in S)=p(1-p)^{|N_F(e_\lambda)|-2}\ge p(1-p)^C>\frac{1}{e(C+1)}.
\]
Consequently, \(\mathbb E(Z_{y,z}\mid y\in S)>\mu_\ell\). By assumption, \(\mu_\ell\ge 18Q\). Let \(m_{y,z}:=\mathbb E(Z_{y,z}\mid y\in S)\). Then we have \(Q\le m_{y,z}/18\). Thus, by Lemma~\ref{lem:chernoff},
\[
\mathbb P(Z_{y,z}<Q\mid y\in S)
\le
\mathbb P\!\left(Z_{y,z}\le \frac{m_{y,z}}{18}\,\middle|\, y\in S\right)
\le
\exp\!\left(-\frac{(17/18)^2m_{y,z}}{2}\right)
\le
\exp(-\mu_\ell/8).
\]
This implies Claim \ref{c4.4}. \hfill$\blacksquare$

\medskip
Note that if \(w\in N_F(e_\lambda)\), then \(\dist_J(y,w)\le L+(2\ell+1)=6\ell+2\). Therefore, if \(E_{y,z}\) and \(E_{y',z'}\) are dependent, then \(\dist_J(y,y')\le12\ell+4\). Since \(\Delta(J)\le D\) and each event $E_{y,z}$ is independent of all but at most $4D^{12\ell+5}$ other events $E_{y',z'}$, we have $\Delta_{\mathcal E}\le 4D^{12\ell+5}$. By the assumption that $\pi \le \frac{1}{e(4D^{12\ell+5} + 1)}$ and Lemma~\ref{LLL}, \(\mathbb P(\text{no }E_{y,z}\text{ occurs})\ge (1-e\pi)^{D|J|}\). Since \(e\pi\le 1/2\), we further have
\((1-e\pi)^{D|J|}\ge \exp(-2e\pi D|J|)\). On the other hand, $|U'\cap S|\sim \mathrm{Bin}(|U'|,p)$ and \(\mathbb E|U'\cap S|=p|U'|\ge pc_0|J|\). Therefore, by Lemma~\ref{lem:chernoff},
\[
\mathbb P\left(|U'\cap S|<\frac12 pc_0|J|\right)
\le \exp(-pc_0|J|/8).
\]
Since \(2eD\pi<pc_0/8\), there exists an outcome for which no event \(E_{y,z}\) occurs and $|U'\cap S|\ge \frac12 pc_0|J|$. Fix such an outcome.

\begin{claim} \label{c4.5}
Let $B_H$ denote the set of $Q$-robustly $a$-branchable vertices of $H$. Then $|B_H|\ge \frac12 pc_0|J|$.
\end{claim}
It suffices to show that $U'\cap S\subseteq B_H$. Let \(y\in U'\cap S\). Since \(y\in U\), we can choose distinct neighbors \(z_1(y),\dots,z_a(y)\in N_J(y)\). For each \(i\in[a]\), define \(M_i(y):=\{y'\in N_H(y): z_i(y)\in V(P_{yy'})\}\). Since no event \(E_{y,z_i(y)}\) occurs, we have \(|M_i(y)|\ge Q\) for every \(i\). For any choice of $y_i\in M_i(y)$, \(P_{yy_1},\dots,P_{yy_a}\) are internally vertex-disjoint and pairwise anticomplete in \(T_y-y\). Since \(N_F(yy_i)\cap S=\{y,y_i\}\), \(P_{yy_i}\) has no vertices or neighbors in \(T_{y_j}\) for $j\neq i$. Thus, these paths are pairwise internally vertex-disjoint and pairwise anticomplete in \(J-y\). Therefore, \(y\) is \(Q\)-robustly \(a\)-branchable in \(H\). This proves Claim \ref{c4.5}.
\hfill$\blacksquare$

\medskip
Since each tree $T_w$ has size $|T_w| \le 2D^{2\ell}$ and at most one edge connects any two distinct trees, we have $\Delta(H) \le \Delta(H^*) \le D|T_w| \le 2D^{2\ell+1} \le D_0$. Consequently, the definition of $D_0$ guarantees that $|B_H| \ge |V(H)|/D_0$. Moreover, since no $E_{y,z}$ occurs, \(\delta(H)\ge Q= 9q^2\). Note that every cycle of length $t$ in $H$ gives a cycle in \(J\) of length at most $Lt$. Since \(g(J)\ge g_{s,\eta,D}\ge L(2m+2)\), the graph $H$ has girth at least $2m+2$. By Lemma~\ref{lem:core} (applied to \(H\) and \(B_H\)), there is a $q$-connected induced subgraph \(H'\subseteq H\) with $|V(H')|\ge 11q^2D_0^2$. Moreover, at least $|V(H')|/(2D_0)$ vertices \(y\in B_H\cap V(H')\) satisfy $d_{H'}(y)\ge d_H(y)-2q^2$; we call these \textit{good} vertices.

Let $\Gamma$ be the graph on the good vertices in which two vertices are adjacent if they are adjacent in \(H'\) or in \(J\). Since \(\Delta(H')\le D_0\) and \(\Delta(J)\le D\le D_0\), we have \(\Delta(\Gamma)\le 2D_0\). There are at least \(|V(H')|/(2D_0)\ge (11/2)q^2D_0\) good vertices, and hence \(\Gamma\) contains an independent set of size at least \(\frac{(11/2)q^2D_0}{2D_0+1}\geq s\), say \(v_1,\dots,v_s\).

\begin{claim}\label{c4.6}
$J$ contains an induced subdivision of $K_s$ whose branch vertices are \(v_1,\dots,v_s\).
\end{claim}

For each \(i\in[s]\), let \(M_1(v_i),\dots,M_{s-1}(v_i)\subseteq N_H(v_i)\) be witness sets in Definition \ref{def:robust}. Since \(d_{H'}(v_i)\ge d_H(v_i)-2q^2\), every witness set contains at least \(Q-2q^2=7q^2\) vertices of \(H'\). Thus, we can injectively assign these sets to the other branch vertices $v_j$ ($j\ne i$) and select pairwise distinct vertices $u_j^{(i)}\in N_{H'}(v_i)$ from them.

Let \(H'':=H'-\{v_1,\dots,v_s\}\). Since \(H'\) is \(q\)-connected, \(H''\) is \((q-s)\)-connected. Note that \(q-s=10(s-1)^2\ge 10\binom{s}{2}\). By Lemma~\ref{lem:TW}, the graph \(H''\) is \(\binom{s}{2}\)-linked. Consequently, there exist pairwise vertex-disjoint paths $Q_{ij}$ in $H''$ connecting $u_j^{(i)}$ and $u_i^{(j)}$ for all $1\le i<j\le s$. For \(1\le i<j\le s\), let \(Y_{ij}\) be the union of the paths \(P_{v_i u_j^{(i)}}\) and \(P_{v_j u_i^{(j)}}\), together with the paths \(P_e\) for all edges \(e\in E(Q_{ij})\). Let \(R_{ij}\) be a shortest \(v_i\),\(v_j\)-path in \(J[V(Y_{ij})]\). Then \(R_{ij}\) is induced.

We claim that \(J\big[\bigcup_{1\le i<j\le s}V(R_{ij})\big]\) is an induced subdivision of \(K_s\). Since $\{v_1,\dots,v_s\}$ is an independent set in $\Gamma$, we only show that there is no edge between distinct paths $R_{ij}$ and $R_{i'j'}$. Suppose such an edge $xy$ exists.
Choose edges \(f_x,f_y\in E(H)\) with \(x\in V(P_{f_x})\) and \(y\in V(P_{f_y})\). If \(f_x\) and \(f_y\) are vertex-disjoint, then the definition of \(H\) implies that \(P_{f_x}\) and \(P_{f_y}\) are anticomplete in \(J\), a contradiction. Hence \(f_x\) and \(f_y\) share an end vertex. Since the paths \(Q_{ij}\) are vertex-disjoint in \(H''\), the only possible shared end vertex is one of the branch vertices, say \(v_i\). Thus \(f_x=v_i u_j^{(i)}\) and \(f_y=v_i u_{j'}^{(i)}\) for
distinct \(j,j'\). By the robustness condition, the corresponding paths are internally vertex-disjoint and anticomplete in \(J-v_i\), again a contradiction. Finally, if $v_r$ is a branch vertex and $e\in E(H)$ is not incident with $v_r$, then $v_r\notin N_F(e)$. Thus, $v_r\notin V(P_e)$ and $v_r$ has no neighbor in $P_e$. Therefore, \(J\big[\bigcup_{1\le i<j\le s}V(R_{ij})\big]\) is an induced subdivision of $K_s$. \hfill$\blacksquare$

\medskip 
By Claim \ref{c4.6}, we obtain a contradiction. This completes the proof of the theorem.
\end{proof}

\begin{corollary}\label{prop:avg-mader}
Let \(g_1:=8\cdot10^6\). For every integer $d\ge4$, every graph $J$ satisfying
\(\Delta(J)<d^{43}\), \(d(J)>d-1+1/20\), and \(g(J)\ge g_1\) contains an induced subdivision of $K_{d+1}$.
\end{corollary}

\begin{proof}
We apply Theorem~\ref{thm:avg-mader} with \(s=d+1\), \(\eta=1/20\), and \(D=d^{43}\). Then \(a=s-1=d\). Let the parameters $q, Q, \alpha, \gamma, C, p, c_0, \mu_\ell$ and $D_0$ be defined as in Theorem \ref{thm:avg-mader}. Take \((\ell,m)=(205,4814)\) if \(d=4\), take \((\ell,m)=(136,3423)\) if \(d=5\), and take \((\ell,m)=(113,5000)\) if \(d\ge6\).
In each case, it is easily verified that $(4\ell+1)(2m+2) < 8\cdot10^6$ and the following conditions hold.
\begin{enumerate}
    \item[(i)] $\mu_\ell > 18Q$ and $\mu_\ell/8 > \max\{1+\log(4D^{12\ell+5}+1), \log(16eD/(pc_0))\}$;
    \item[(ii)] $2(q^2-1)^m \ge 11q^2D_0^2$.
\end{enumerate}
By Theorem~\ref{thm:avg-mader}, we can find an induced $K_{d+1}$-subdivision in $J$.
\end{proof}

\section{Proof of Theorem \ref{thm:main}}\label{sec:main-proof}
In this section, we prove our main result, i.e., Theorem \ref{thm:main}. Our approach is based on the techniques introduced by Gir\~{a}o and Hunter.
\begin{proof}
Let \(g_0:=8\cdot10^6\) and let $k\ge3$ be fixed. Let $G$ be a counterexample of Theorem~\ref{thm:main}. Let \(d:=\max_{\varnothing\neq X\subseteq V(G)}\delta(G[X])\) and pass to an induced subgraph attaining this maximum, again denoted by $G$. Since \(d\ge k\), any induced subdivision of \(K_{d+1}\) contains an induced subdivision of \(K_{k+1}\). Hence \(G\) contains no induced subdivision of \(K_{d+1}\). Let $n:=|G|$. Suppose \(d=3\). Then \(G\) does not contain an induced subdivision of $K_4$ and \(g(G)\ge5\). By Theorem~\ref{thm:TV-ISK4}, \(G\) has a vertex of degree at most \(2\), a contradiction to \(\delta(G)=3\). Hence \(d\ge4\).

Let \(B:=\{v\in V(G):d_G(v)\ge d^{43}\}\). Let \(\beta:=10^{-12}\). Since $G$ is $d$-degenerate, the sum of the degrees of $V(G)$ is at most $2dn$, so \(|B|\le 2n/d^{42}\). Let $A:=\{v\in V(G)\setminus B:d_B(v)\ge2\}$. If $|A|>60d^2|B|$, then by Lemma~\ref{lem:unbalanced}, we obtain an induced subdivision of $K_{d+1}$, a contradiction. Thus, we have $|A|\le 60d^2|B|$. Let $A':=\{v\in V(G)\setminus B:d_B(v)=1\}$. We consider the following two cases.

\medskip \noindent 
\textbf{Case 1. $|A'|\ge 81n/100$.}

\medskip \noindent By Lemma~\ref{lem:one-sided} (applied with $X=A'$ and $B_0=B$), we obtain disjoint sets $X'\subseteq A'$ and $Y\subseteq V(G)\setminus X'$ such that $|Y|\ge \beta n/(d^3(d+1))$, $Y$ is independent, $2\le |N(y)\cap X'|\le 3\cdot 10^7 d^4$ for all $y\in Y$, and $d_G(x')\le 800d$ for all $x'\in X'$. Set $Y_0:=Y\setminus(A\cup B)$. Since \(|A|+|B|\le (60d^2+1)|B|\le 2n(60d^2+1)/d^{42}<\beta n/(2d^3(d+1))\), we have $|Y_0|\ge \beta n/(2d^3(d+1))$.

Construct \(Y_1\subseteq Y_0\) greedily as follows. When a vertex \(y\in Y_0\) is chosen, fix two distinct vertices \(x_1^y,x_2^y\in N(y)\cap X'\), add \(y\) to \(Y_1\), and then delete every vertex of \(Y_0\) adjacent to \(x_1^y\) or \(x_2^y\). Since each \(x_i^y\) has degree at most \(800d\), we delete at most \(1600d\) vertices in each step. Hence \(|Y_1|\ge |Y_0|/(1600d)\). By construction, the triples \(P_y:=\{y,x_1^y,x_2^y\}\) for \(y\in Y_1\) are pairwise disjoint. Define a graph \(L\) on \(Y_1\) by joining \(y\) to \(y'\), with \(y\) chosen earlier than \(y'\), if some vertex in \(P_y\) is adjacent to \(x_1^{y'}\) or \(x_2^{y'}\). For each fixed \(y'\in Y_1\), there are at most \(1600d\) such earlier neighbors, so \(L\) is \(1600d\)-degenerate. Therefore \(L\) has an independent set \(Y_2\subseteq Y_1\) with $|Y_2|\ge \frac{|Y_0|}{1600d(1600d+1)}$. For any two distinct vertices $y_1,y_2 \in Y_2$, the triples \(P_{y_1}\) and \(P_{y_2}\) are anticomplete in \(G\).

For each \(y\in Y_2\), let \(b_1^y\) and \(b_2^y\) be the unique neighbors of $x_1^y$ and \(x_2^y\) in $B$, respectively. Then \(b_1^y\neq b_2^y\) and \(b_1^y b_2^y\notin E(G)\) by the girth assumption. Fix a \(d\)-degenerate ordering of \(B\), choose each vertex of \(B\) independently at random with probability \(q_B:=1/(2d)\), and let \(R\) be the set of chosen vertices with no chosen right-neighbor. Then \(R\) is independent. Call \(y\in Y_2\) \textit{good} if \(b_1^y,b_2^y\in R\) and \(N(y)\cap R=\varnothing\). Since \(y\notin A\cup B\), it has at most one neighbor in \(B\), and \(y\) is adjacent to neither \(b_1^y\) nor \(b_2^y\). Hence, since $q_B=1/(2d)$ and $d\geq 4$, a direct lower bound yields
$\Prb(y\text{ is good})\ge \frac{1}{4d^2}\left(1-\frac{1}{2d}\right)^{2d+1} \ge \frac{1}{14d^2}$.
Let $Z$ be the number of good vertices. Thus, we have
$\E Z\ge \frac{|Y_2|}{14d^2} \ge  \frac{\beta n}{44800d^6(1600d+1)(d+1)}>\frac{2n}{d^{42}}$
as $d\ge4$. On the other hand, \(\E(d|R|)\le dq_B|B|=|B|/2\le {n}/{d^{42}}\). Therefore there exists some choice of $R$ that satisfies \(Z>d|R|\).

Define an auxiliary graph \(H\) on vertex set \(R\) by adding the edge \(b_1^yb_2^y\) for each good \(y\in Y_2\). Note that \(H\) is simple: distinct good vertices yield distinct edges of \(H\). Indeed, if two distinct good vertices \(y,y'\) yield the same edge \(uv\), then the two anticomplete triples \(P_y\) and \(P_{y'}\), together with \(u\) and \(v\), would form an $8$-cycle in \(G\). Thus \(|E(H)|=Z>d|R|\) and \(d(H)>2d\). By Lemma~\ref{lem:avg-core}, \(H\) contains an induced subgraph \(H'\) with \(\delta(H')\ge d\).

For each edge \(e=uv\in E(H')\), let \(y_e\) be the unique good vertex corresponding to \(e\), and let \(S_e:=u\,x_u^e\, y_e\, x_v^e\, v\) be the corresponding \(u\),\(v\)-path, where \(x_u^e,x_v^e\in X'\) are the two chosen neighbors of \(y_e\) adjacent to \(u\) and \(v\), respectively. It is not hard to see that each \(S_e\) is induced. Moreover, if \(e\neq f\), then the internal vertices of \(S_e\) and \(S_f\) lie in two distinct sets \(P_{y_e}\) and \(P_{y_f}\), which are disjoint and anticomplete. Hence, the internal vertices of \(S_e\) and \(S_f\) are internally disjoint and anticomplete. If \(C\) is a cycle of length \(t\) in \(H'\), then replacing every edge \(e\in E(C)\) by \(S_e\) gives a cycle of length \(4t\) in \(G\). Thus \(g(H')\ge 27\), for otherwise \(G\) would contain a cycle of length less than \(108\). Therefore Theorem~\ref{KO-topological-clique} implies a subdivision \(T\) of \(K_{d+1}\) in \(H'\). Replacing every edge \(e\in E(T)\) by \(S_e\) gives an induced subdivision of \(K_{d+1}\) in \(G\), a contradiction.

\medskip \noindent 
\textbf{Case 2. $|A'|<81n/100$.}

\medskip \noindent 
Let \(G':=G-B\). Then \(\Delta(G')<d^{43}\) and \(g(G')\ge g(G)\). Every vertex of \(V(G)\setminus(A\cup A'\cup B)\) has degree at least \(d\) in \(G'\), while every vertex of \(A'\) has degree at least \(d-1\) in \(G'\). Therefore
\(
  2e(G')\ge d\bigl(n-|A|-|A'|-|B|\bigr)+(d-1)|A'|
  =d(n-|A|-|B|)-|A'|.
\)
Since \(|A'|<81n/100\) and \(|V(G')|\le n\), we obtain
\(d(G')>d-81/100-d(|A|+|B|)/n\). Thus, we have
\(
  d\frac{|A|+|B|}{n}\le d(60d^2+1)\frac{|B|}{n}\le \frac{2(60d^2+1)}{d^{41}}<1-81/100-1/20,
\)
and hence \(d(G')>d-1+1/20\). Now we apply Corollary~\ref{prop:avg-mader} to \(G'\). We have that \(G'\) contains an induced subdivision of \(K_{d+1}\), a contradiction. This completes the proof.
\end{proof}

\section{Concluding remarks}\label{sec:concluding}

The constant $8 \cdot 10^6$ in Theorem~\ref{thm:main} has not been strictly optimized, and the optimal value of $g_0$ remains unknown. However, we can show that $g_0\geq 7$. Let $L(PG(2, q))$ denote the incidence graph of the points and lines of the projective plane $PG(2, q)$. For any odd prime power $q \ge 5$, we can show that $L(PG(2, q))$ contains no induced subdivision of $K_{q+2}$. Indeed, the induced condition, together with the structural constraints of the subdivision,
forces all $q+2$ branch vertices to lie in the same partition class of the bipartite incidence graph. In $PG(2, q)$, this corresponds to a set of $q+2$ points such that no three are collinear, forming a $(q+2)$-arc. However, a classical theorem of Segre~\cite{Segre1955} establishes that $PG(2, q)$ contains no arc of size $q+2$ when $q$ is odd. Consequently, we have the following.
\begin{proposition}
For infinitely many integers \(d\), there exists a \(d\)-regular graph of girth six with no induced subdivision of \(K_{d+1}\).
\end{proposition}

If one considers graphs with large minimum degree, the girth requirement can be significantly reduced. For instance, using the same argument, we can show that if $d \ge 8500$, then every graph with minimum degree at least $d$ and girth at least $21000$ contains an induced subdivision of $K_{d+1}$.

Note that, in our proof of Theorem \ref{thm:main}, all induced subdivisions we find are proper (no two branch vertices are adjacent). This immediately leads to the following general statement, extending the result from complete graphs to the class of graphs with maximum degree at most $d$.
\begin{theorem}
For every fixed graph \(H\), every graph \(G\) with \(\delta(G)\ge |V(H)|-1\) and sufficiently large girth contains an induced subdivision of \(H\).
\end{theorem}

\end{document}